\documentclass[a4paper,11pt]{article}
\usepackage{bm,mathrsfs,amsmath,amsthm,amssymb,amscd,longtable,array,graphicx}
\newcommand{\nc}{\newcommand}
\nc{\rnc}{\renewcommand}
\nc{\nn}{\nonumber}
\nc{\der}{{\partial}}
\rnc{\Im}{{\rm{Im}\,}}
\rnc{\Re}{{\rm{Re}\,}}
\nc{\db}{\displaybreak[0]\\}
\nc{\bra}{\langle}
\nc{\ket}{\rangle}
\nc{\bs}{\boldsymbol}

\newtheorem{theorem}{Theorem}[section]

\newtheorem{proposition}[theorem]{Proposition}

\theoremstyle{definition}

\numberwithin{equation}{section}

\numberwithin{equation}{section}

\begin{document}%
%
\title{A duality formula between elliptic determinants}

\author{
Kohei Motegi \thanks{E-mail: kmoteg0@kaiyodai.ac.jp}
\\\\
{\it Faculty of Marine Technology, Tokyo University of Marine Science and Technology,}\\
 {\it Etchujima 2-1-6, Koto-Ku, Tokyo, 135-8533, Japan} \\
\\\\
\\
}

\date{\today}

\maketitle

\begin{abstract}
We prove a duality formula between two elliptic determinants.
We present a proof which is a variant of the Izergin-Korepin method
which is a method originally introduced
to analyze and compute partition functions of integrable lattice models.
\end{abstract}

\section{Introduction}
Elliptic special functions is an active area of research in recent years.
Investigating formulas for special determinants and Pfaffians
whose matrix elements are given in terms of elliptic functions
is one of the most fundamental subjects,
and there are developments on the
evaluations of the elliptic determinants and Pfaffians in recent years.
For example, factorization formulas for
various elliptic determinants which are analogues, extensions
and variants of the classical
(elliptic) Cauchy determinant (Frobenius determinant)
formula \cite{Cauchy,Frob,Kr} were found
\cite{Has,TV,War,RosSch,Schpath,BK,Rains}.
The Pfaffian analogues of the elliptic Cauchy determinant
formula were also found
\cite{Rains,Okadapfaffian,Roselldet,Rosellpfaffian}, which
generalizes the classical Pfaffian formula \cite{Schur}.

Another interesting subject is to find transformation formulas
between two elliptic determinants or two different Pfaffians
which look totally different at first sight.
See \cite{RosRokko,RosSchtwo} for seminal works on this subject.
Recently, we found a transformation formula between two elliptic Pfaffians
by studying the partition functions of an elliptic integrable model
in two ways \cite{MotegiPfaffian}.
A special case of the transformation formula can be proved 
easily by combining factorization formulas for two elliptic Pfaffians
by Rains \cite{Rains} and Rosengren \cite{Roselldet}.
However, beyond that special point where no
factorized expressions are known,
it seems not so easy to prove the transformation formula.
A similar situation has already appeared in the work by
Rosengren \cite{RosRokko}, in which he proved a duality between
two elliptic determinants which is an elliptic analogue of a duality
discovered by Rosengren-Schlosser \cite{RosSchtwo}.
A special case can be proved by using
the elliptic determinant evaluations by Warnaar \cite{War},
but no factorization formulas are known in general case
in which giving a proof is not so easy.
Rosengren gives a proof of his duality formula \cite{RosRokko}
by using the elliptic Jackson summation formula.
See
\cite{DJKMO,FT,Rosroot,CosGus,RainsBC,Rainselliptic,Spi,vDSpi,Rossummation,ItoNoumi,RW}
for examples on seminal works and various extensions of
the elliptic Jackson summation formulas
and corresponding elliptic integral formulas.
In general, when we do not have
or do not know whether there are factorized expressions,
it is not easy to prove the transformation formula
between determinants or Pfaffians.

In this paper, we prove another
transformation formula for two elliptic determinants
which do not seem to have factorized expressions.
We prove the following theorem.

\begin{theorem} \label{identitytheorem}
The following identity between two elliptic determinants holds:
\begin{align}
&\mathrm{det}_N (X_N(z_1,\dots,z_N|w_1,\dots,w_N|h))
\nonumber \\
=&\displaystyle \frac{[h]\prod_{j=1}^N [h-j/2+1]}
{[h-N/2 ]\prod_{j=1}^N [h+(N+1)/2-j]}
\prod_{1 \le j < k \le N} \frac{[w_k-w_j+1/2]}{[w_k-w_j]}
\nonumber \\
&\times \mathrm{det}_N (Y_N(z_1,\dots,z_N|w_1,\dots,w_N|h)),
\label{ellipticidentity}
\end{align}
where $X_N(z_1,\dots,z_N|w_1,\dots,w_N|h)$ and
$Y_N(z_1,\dots,z_N|w_1,\dots,w_N|h)$ are $N \times N$ matrices
whose matrix elements are given by
\begin{align}
&X_N(z_1,\dots,z_N|w_1,\dots,w_N|h)_{jk} \nonumber \\
=&[h+(j-N)/2+z_k+w_j] \prod_{\ell=1}^{j-1} [w_\ell+z_k+1/2]
\prod_{\ell=j+1}^{N} [w_\ell+z_k]
\prod_{\ell=1}^{N} [w_\ell-z_k] \nonumber \\
&-[h+(j-N)/2-z_k+w_j] \prod_{\ell=1}^{j-1} [w_\ell-z_k+1/2]
\prod_{\ell=j+1}^{N} [w_\ell-z_k]
\prod_{\ell=1}^{N} [w_\ell+z_k], \label{matrixelementsone} \\
&Y_N(z_1,\dots,z_N|w_1,\dots,w_N|h)_{jk} \nonumber \\
=&[h+(N-1)/2+w_k+z_j] \prod_{ \substack{ \ell=1 \\ \ell \neq j}}^{N} [z_\ell+w_k]
\prod_{\ell=1}^{N} [z_\ell-w_k] \nonumber \\
&-[-h-(N-1)/2-w_k+z_j] \prod_{ \substack{ \ell=1 \\ \ell \neq j}}^{N} [z_\ell-w_k]
\prod_{\ell=1}^{N} [z_\ell+w_k], \label{matrixelementstwo}
\end{align}
for $j,k=1,\dots,N$.
Here, $[u]$ is the theta function $[u]=H(\pi i u)$
where $H(u)$ is given by
\begin{align}
H(u)=2 \sinh u
\prod_{j=1}^\infty (1-2{\bf q}^{2j} \cosh 2u+{\bf q}^{4j})(1-{\bf q}^{2j}),
\end{align}
where ${\bf q}$ is the elliptic nome $(0 < {\bf q} < 1)$.

\end{theorem}
We prove Theorem \ref{identitytheorem} in this paper.
The determinants $\mathrm{det}_N (X_N(z_1,\dots,z_N|w_1,\dots,w_N|h))$
and $\mathrm{det}_N (Y_N(z_1,\dots,z_N|w_1,\dots,w_N|h))$
in the Theorem do not seem to factorize,
and we present a proof which works
for determinants which do not seem to
have factorized expressions.
The proof is inspired by and can be regarded as a variant of the
Izergin-Korepin method \cite{Ko,Iz}
in the field of quantum integrable models.
The Izergin-Korepin method was initiated by Korepin \cite{Ko},
which he introduced a way to characterize the domain wall boundary
partition functions of the  $U_q(\widehat{sl_2})$ six-vertex model
\cite{Dr,J,FRT,Baxter,KBI,LW}
which uniquely define them.
Izergin \cite{Iz} later found a determinant form which satisfies
all the properties listed by Korepin,
and is now called as the Izergin-Korepin determinant, which today have many
applications to other branches of mathematics and further investigations,
such as the enumeration of the alternating sign matrices
\cite{Br,Ku1,Ku2,Okada,CP,BWZ} and the thermodynamic limit \cite{KZ}.
The Izergin-Korepin method was also extended to various boundary conditions
\cite{Ku2,Tsuchiya} and other classes of partition functions
such as the scalar products \cite{Wheeler}
and the wavefunctions \cite{MotegiIzerginKorepin}.

The original Izergin-Korepin method is a way to prove
identities between partition functions
which are functions construced by the $R$-matrices
of quantum integrable models,
and explicit determinants, Pfaffians
or symmetric functions.
The idea of the Izergin-Korepin method
can also be used to prove identities between two functions
which look at first sight totally different.
The idea of the method do not have to be restricted to
the computations of partition functions,
and we present one such application in this paper.

This paper is organized as follows.
In the next section, we present some properties
of the theta functions which will be used in this paper,
and check the simplest nontrivial example, i.e.,
the case $N=2$.
In section 3, we present a proof which is inspired
the Izergin-Korepin method.
Section 4 is devoted to the conclusion of this paper.

\section{Preliminaries and the simplest nontrivial example}
In this section,
we first list the properties of theta functions used in this paper.
One of the most fundamental properties about theta functions
is the quasi-periodicities
\begin{align}
[u+1]&=-[u], \label{qpone} \\
[u-i \log ({\bf q})/\pi]&=-{\bf q}^{-1}
\exp  (-2 \pi i u) [u] \label{qptwo}.
\end{align}
Using \eqref{qpone} and the fact that $[u]$ is an odd function
$[-u]=-[u]$, we get
\begin{align}
[u+1/2]=[-u+1/2], \label{usethisproperty}
\end{align}
which is an important property used in this paper.
Another important property is the addition formula for the theta functions
\begin{align}
[u+x][u-x][v+y][v-y]-[v+x][v-x][u+y][u-y]-[x+y][x-y][u+v][u-v]=0.
\label{additionformula}
\end{align}
We will use the above properties
repeatedly to check the simplest nontrivial example
of Theorem \ref{identitytheorem}.

For the proof of Theorem \ref{identitytheorem},
besides the above fundamental properties for the theta functions,
the following notions and properties
about the elliptic polynomials \cite{PRS,FSfelderhof}
is crucial.

A character is a group homomorphism
$\chi$ from multiplicative groups
$\Gamma=\mathbf{Z}+\tau \mathbf{Z}$ to $\mathbf{C}^\times$.
For each character $\chi$ and positive integer $n$, an $n$-dimensional space $\Theta_n(\chi)$
is a set of holomorphic functions $\phi(y)$ on $\mathbf{C}$
satisfying the quasiperiodicities
\begin{align}
\phi(y+1)&=\chi(1) \phi(y), \label{propertyuseone} \\
\phi(y+\tau)&=\chi(\tau) e^{-2 \pi i ny-\pi i n \tau}\phi(y).
\label{propertyusetwo}
\end{align}
The elements of the space $\Theta_n(\chi)$ are called elliptic polynomials.
The space $\Theta_n(\chi)$ is $n$-dimensional \cite{PRS,FSfelderhof},
and the following fact holds for the elliptic polynomials:
\begin{proposition} \cite{PRS,FSfelderhof} \label{propositionelliptic}
Suppose there are two elliptic polynomials $P(y)$ and $Q(y)$
in $\Theta_n(\chi)$, where $\chi(1)=(-1)^n$ and $\chi(\tau)=(-1)^n e^\alpha$.
If these two polynomials are equal at $n$ points $y_j$, $j=1,\dots,n$, satisfying
$y_j-y_k \not\in \Gamma$ and $\sum_{k=1}^N y_k-\alpha \not\in \Gamma$, that is, $P(y_j)=Q(y_j)$,
then the two polynomials are exactly the same: $P(y)=Q(y)$.
\end{proposition}

These properties played important roles for developing
methods for elliptic quantum integrable models, such as 
the separation of the variables method
and the Izergin-Korepin method.
For example, it was used to analyze and compute the explicit forms
of the domain wall boundary partition functions of
the Andrews-Baxter-Forrester model \cite{ABFfelderhof}.
See Refs. \cite{PRS},
\cite{FSfelderhof}, \cite{Ros}, and \cite{YZ} for examples for seminal works
of the developments.
We use this property in the next section
to prove Theorem \ref{identitytheorem}. 
\\

In the end of this section,
let us check Theorem \ref{identitytheorem} by the simplest nontrivial
case $N=2$ by elementary manipulations
(the case $N=1$ is trivial to check).
Using \eqref{usethisproperty} and \eqref{additionformula},
one can show the following four relations.
\begin{align}
&[h-1/2+z_1+w_1][w_2+z_1][h+z_2+w_2][w_1+z_2+1/2] \nonumber \\
-&[h-1/2+z_2+w_1][w_2+z_2][h+z_1+w_2][w_1+z_1+1/2] \nonumber \\
=&-[h+1/2+z_1+w_1][w_2+z_1][h+z_2+w_2][w_1+z_2+1/2] \nonumber \\
+&[h+1/2+z_2+w_1][w_2+z_2][h+z_1+w_2][w_1+z_1+1/2] \nonumber \\
=&-[h][w_2-w_1+1/2][h+1/2+z_1+z_2+w_1+w_2][z_1-z_2], \label{first} \\
&[h-1/2+z_1+w_1][w_2+z_1][h-z_2+w_2][w_1-z_2+1/2] \nonumber \\
-&[h-1/2-z_2+w_1][w_2-z_2][h+z_1+w_2][w_1+z_1+1/2] \nonumber \\
=&-[h+1/2+z_1+w_1][w_2+z_1][h-z_2+w_2][w_1-z_2+1/2] \nonumber \\
+&[h+1/2-z_2+w_1][w_2-z_2][h+z_1+w_2][w_1+z_1+1/2] \nonumber \\
=&-[h][w_2-w_1+1/2][h+1/2+z_1-z_2+w_1+w_2][z_1+z_2], \label{second} \\
&[h-1/2-z_1+w_1][w_2-z_1][h+z_2+w_2][w_1+z_2+1/2] \nonumber \\
-&[h-1/2+z_2+w_1][w_2+z_2][h-z_1+w_2][w_1-z_1+1/2] \nonumber \\
=&-[h+1/2-z_1+w_1][w_2-z_1][h+z_2+w_2][w_1+z_2+1/2] \nonumber \\
+&[h+1/2+z_2+w_1][w_2+z_2][h-z_1+w_2][w_1-z_1+1/2] \nonumber \\
=&[h][w_2-w_1+1/2][h+1/2-z_1+z_2+w_1+w_2][z_1+z_2], \label{third} \\
&[h-1/2-z_1+w_1][w_2-z_1][h-z_2+w_2][w_1-z_2+1/2] \nonumber \\
-&[h-1/2-z_2+w_1][w_2-z_2][h-z_1+w_2][w_1-z_1+1/2] \nonumber \\
=&-[h+1/2-z_1+w_1][w_2-z_1][h-z_2+w_2][w_1-z_2+1/2] \nonumber \\
+&[h+1/2-z_2+w_1][w_2-z_2][h-z_1+w_2][w_1-z_1+1/2] \nonumber \\
=&[h][w_2-w_1+1/2][h+1/2-z_1-z_2+w_1+w_2][z_1-z_2]. \label{fourth}
\end{align}
Using \eqref{first}, \eqref{second}, \eqref{third} and \eqref{fourth},
one can rewrite $\mathrm{det}_2 (X_2(z_1,z_2|w_1,w_2|h))$ as
\begin{align}
&\mathrm{det}_2 (X_2(z_1,z_2|w_1,w_2|h)) \nonumber \\
=&[w_1-z_1][w_2-z_1][w_1-z_2][w_2-z_2] \nonumber \\
\times&([h-1/2+z_1+w_1][w_2+z_1][h+z_2+w_2][w_1+z_2+1/2] \nonumber \\
-&[h-1/2+z_2+w_1][w_2+z_2][h+z_1+w_2][w_1+z_1+1/2]) \nonumber \\
-&[w_1-z_1][w_2-z_1][w_1+z_2][w_2+z_2] \nonumber \\
\times&([h-1/2+z_1+w_1][w_2+z_1][h-z_2+w_2][w_1-z_2+1/2] \nonumber \\
-&[h-1/2-z_2+w_1][w_2-z_2][h+z_1+w_2][w_1+z_1+1/2]) \nonumber \\
-&[w_1+z_1][w_2+z_1][w_1-z_2][w_2-z_2] \nonumber \\
\times&([h-1/2-z_1+w_1][w_2-z_1][h+z_2+w_2][w_1+z_2+1/2] \nonumber \\
-&[h-1/2+z_2+w_1][w_2+z_2][h-z_1+w_2][w_1-z_1+1/2]) \nonumber \\
+&[w_1+z_1][w_2+z_1][w_1+z_2][w_2+z_2] \nonumber \\
\times&([h-1/2-z_1+w_1][w_2-z_1][h-z_2+w_2][w_1-z_2+1/2] \nonumber \\
-&[h-1/2-z_2+w_1][w_2-z_2][h-z_1+w_2][w_1-z_1+1/2]) \nonumber \\
=&[w_1-z_1][w_2-z_1][w_1-z_2][w_2-z_2] \nonumber \\
\times&
(-[h][w_2-w_1+1/2][h+1/2+z_1+z_2+w_1+w_2][z_1-z_2])
\nonumber \\
-&[w_1-z_1][w_2-z_1][w_1+z_2][w_2+z_2] \nonumber \\
\times&
(-[h][w_2-w_1+1/2][h+1/2+z_1-z_2+w_1+w_2][z_1+z_2])
\nonumber \\
-&[w_1+z_1][w_2+z_1][w_1-z_2][w_2-z_2] \nonumber \\
\times&
[h][w_2-w_1+1/2][h+1/2-z_1+z_2+w_1+w_2][z_1+z_2]
\nonumber \\
+&[w_1+z_1][w_2+z_1][w_1+z_2][w_2+z_2] \nonumber \\
\times&
[h][w_2-w_1+1/2][h+1/2-z_1-z_2+w_1+w_2][z_1-z_2]
\nonumber \\
=&[h][w_2-w_1+1/2] \nonumber \\
\times&(
-[w_1-z_1][w_1-z_2][w_2-z_1][w_2-z_2][z_1-z_2][h+1/2+z_1+z_2+w_1+w_2]
\nonumber \\
+&[w_1-z_1][w_1+z_2][w_2-z_1][w_2+z_2][z_1+z_2][h+1/2+z_1-z_2+w_1+w_2]
\nonumber \\
-&[w_1+z_1][w_1-z_2][w_2+z_1][w_2-z_2][z_1+z_2][h+1/2-z_1+z_2+w_1+w_2]
\nonumber \\
+&[w_1+z_1][w_1+z_2][w_2+z_1][w_2+z_2][z_1-z_2][h+1/2-z_1-z_2+w_1+w_2]
), \label{forcomparisonone}
\end{align}
which is a simplification of the left hand side of
\eqref{ellipticidentity} for the case $N=2$.

Let us next examine the right hand side.
The right hand side of \eqref{ellipticidentity}
for the case $N=2$ is $\displaystyle \frac{[h][w_2-w_1+1/2]}{[h+1/2][w_2-w_1]}
\mathrm{det}_2 (Y_2(z_1,z_2|w_1,w_2|h))$.
Using the addition formula \eqref{additionformula},
one can show the following four relations.
\begin{align}
&[h+1/2+w_1+z_1][z_2+w_1][h+1/2+w_2+z_2][z_1+w_2] \nonumber \\
-&[h+1/2+w_1+z_2][z_1+w_1][h+1/2+w_2+z_1][z_2+w_2] \nonumber \\
=&[h+1/2][z_1-z_2][h+1/2+z_1+z_2+w_1+w_2][w_1-w_2], \label{rightfirst} \\
&[h+1/2+w_1+z_1][z_2-w_1][h+1/2+w_2-z_2][z_1+w_2] \nonumber \\
-&[h+1/2+w_2+z_1][z_2-w_2][h+1/2+w_1-z_2][z_1+w_1] \nonumber \\
=&[h+1/2][z_1+z_2][h+1/2+z_1-z_2+w_1+w_2][w_2-w_1], \label{rightsecond} \\
&[h+1/2+w_1-z_1][z_2+w_1][h+1/2+w_2+z_2][z_1-w_2] \nonumber \\
-&[h+1/2+w_2-z_1][z_2+w_2][h+1/2+w_1+z_2][z_1-w_1] \nonumber \\
=&[h+1/2][z_1+z_2][h+1/2-z_1+z_2+w_1+w_2][w_1-w_2], \label{rightthird} \\
&[h+1/2+w_1-z_1][z_2-w_1][h+1/2+w_2-z_2][z_1-w_2] \nonumber \\
-&[h+1/2+w_2-z_1][z_2-w_2][h+1/2+w_1-z_2][z_1-w_1] \nonumber \\
=&[h+1/2][z_2-z_1][h+1/2-z_1-z_2+w_1+w_2][w_1-w_2]. \label{rightfourth}
\end{align}
Using the four relations
\eqref{rightfirst}, \eqref{rightsecond}, \eqref{rightthird} and
\eqref{rightfourth},
one can simplify \\
$\displaystyle \frac{[h][w_2-w_1+1/2]}{[h+1/2][w_2-w_1]}
\mathrm{det}_2 (Y_2(z_1,z_2|w_1,w_2|h))$ as
\begin{align}
&\displaystyle \frac{[h][w_2-w_1+1/2]}{[h+1/2][w_2-w_1]}
\mathrm{det}_2 (Y_2(z_1,z_2|w_1,w_2|h))=
\frac{[h][w_2-w_1+1/2]}{[h+1/2][w_2-w_1]} \nonumber \\
\times&\{
[z_2-w_2][z_1-w_2][z_2-w_1][z_1-w_1] (
[h+1/2+w_1+z_1][z_2+w_1][h+1/2+w_2+z_2][z_1+w_2] \nonumber \\
-&[h+1/2+w_1+z_2][z_1+w_1][h+1/2+w_2+z_1][z_2+w_2]) \nonumber \\
+&[z_1-w_1][z_1-w_2][z_2+w_1][z_2+w_2] (
[h+1/2+w_1+z_1][z_2-w_1][h+1/2+w_2-z_2][z_1+w_2] \nonumber \\
-&[h+1/2+w_2+z_1][z_2-w_2][h+1/2+w_1-z_2][z_1+w_1]) \nonumber \\
+&[z_1+w_1][z_1+w_2][z_2-w_1][z_2-w_2] (
[h+1/2+w_1-z_1][z_2+w_1][h+1/2+w_2+z_2][z_1-w_2] \nonumber \\
-&[h+1/2+w_2-z_1][z_2+w_2][h+1/2+w_1+z_2][z_1-w_1]) \nonumber \\
+&[z_1+w_1][z_2+w_2][z_1+w_2][z_2+w_1] (
[h+1/2+w_1-z_1][z_2-w_1][h+1/2+w_2-z_2][z_1-w_2] \nonumber \\
-&[h+1/2+w_2-z_1][z_2-w_2][h+1/2+w_1-z_2][z_1-w_1]) \} \nonumber \\
=&\frac{[h][w_2-w_1+1/2]}{[h+1/2][w_2-w_1]} \nonumber \\
\times&\{
[z_2-w_2][z_1-w_2][z_2-w_1][z_1-w_1] \nonumber \\
\times&[h+1/2][z_1-z_2][h+1/2+z_1+z_2+w_1+w_2][w_1-w_2]
\nonumber \\
+&[z_1-w_1][z_1-w_2][z_2+w_1][z_2+w_2] \nonumber \\
\times&[h+1/2][z_1+z_2][h+1/2+z_1-z_2+w_1+w_2][w_2-w_1]
\nonumber \\
+&[z_1+w_1][z_1+w_2][z_2-w_1][z_2-w_2] \nonumber \\
\times&[h+1/2][z_1+z_2][h+1/2-z_1+z_2+w_1+w_2][w_1-w_2]
\nonumber \\
+&[z_1+w_1][z_2+w_2][z_1+w_2][z_2+w_1] \nonumber \\
\times&[h+1/2][z_2-z_1][h+1/2-z_1-z_2+w_1+w_2][w_1-w_2]
\} \nonumber
\end{align}
\begin{align}
=&[h][w_2-w_1+1/2] \nonumber \\
\times&(
-[w_1-z_1][w_1-z_2][w_2-z_1][w_2-z_2][z_1-z_2][h+1/2+z_1+z_2+w_1+w_2]
\nonumber \\
+&[w_1-z_1][w_1+z_2][w_2-z_1][w_2+z_2][z_1+z_2][h+1/2+z_1-z_2+w_1+w_2]
\nonumber \\
-&[w_1+z_1][w_1-z_2][w_2+z_1][w_2-z_2][z_1+z_2][h+1/2-z_1+z_2+w_1+w_2]
\nonumber \\
+&[w_1+z_1][w_1+z_2][w_2+z_1][w_2+z_2][z_1-z_2][h+1/2-z_1-z_2+w_1+w_2]
). \label{forcomparisontwo}
\end{align}
Since the simplifications
\eqref{forcomparisonone} of
$\mathrm{det}_2 (X_2(z_1,z_2|w_1,w_2|h))$
and \eqref{forcomparisontwo} of
\\
$
\displaystyle \frac{[h][w_2-w_1+1/2]}{[h+1/2][w_2-w_1]}
\mathrm{det}_2 (Y_2(z_1,z_2|w_1,w_2|h))$
are the same,
one has checked that
\begin{align}
\mathrm{det}_2 (X_2(z_1,z_2|w_1,w_2|h))
=\displaystyle \frac{[h][w_2-w_1+1/2]}{[h+1/2][w_2-w_1]}
\mathrm{det}_2 (Y_2(z_1,z_2|w_1,w_2|h)),
\end{align}
holds.

\section{Proof}
In this section, we prove
Theorem \ref{identitytheorem}.
We prove the following equivalent
theorem which both hand sides of
\eqref{ellipticidentity} in Theorem \ref{identitytheorem}
are multiplied by $\displaystyle \prod_{1 \le j < k \le N} \frac{[z_j-z_k+1/2]}{[z_j-z_k]}$.
\begin{theorem} \label{identitytheoremprove}
The following identity between two elliptic determinants holds:
\begin{align}
&
\prod_{1 \le j < k \le N} \frac{[z_j-z_k+1/2]}{[z_j-z_k]}
\mathrm{det}_N (X_N(z_1,\dots,z_N|w_1,\dots,w_N|h))
\nonumber \\
=&\displaystyle \frac{[h]\prod_{j=1}^N [h-j/2+1]}
{[h-N/2 ]\prod_{j=1}^N [h+(N+1)/2-j]}
\prod_{1 \le j < k \le N} \frac{[z_j-z_k+1/2][w_k-w_j+1/2]}{[z_j-z_k][w_k-w_j]}
\nonumber \\
&\times \mathrm{det}_N (Y_N(z_1,\dots,z_N|w_1,\dots,w_N|h)).
\label{ellipticidentityprove}
\end{align}
\end{theorem}

\begin{proof}

Let us denote the left hand side and right hand side of
\eqref{ellipticidentityprove} as
$L_N(z_1,\dots,z_N|w_1,\dots,w_N|h)$
and $R_N(z_1,\dots,z_N|w_1,\dots,w_N|h)$ respectively.
\begin{align}
&L_N(z_1,\dots,z_N|w_1,\dots,w_N|h) \nonumber \\
=&
\prod_{1 \le j < k \le N} \frac{[z_j-z_k+1/2]}{[z_j-z_k]}
\mathrm{det}_N (X_N(z_1,\dots,z_N|w_1,\dots,w_N|h)), \label{leftexpression} \\
&R_N(z_1,\dots,z_N|w_1,\dots,w_N|h) \nonumber \\
=&\displaystyle \frac{[h]\prod_{j=1}^N [h-j/2+1]}
{[h-N/2 ]\prod_{j=1}^N [h+(N+1)/2-j]}
\prod_{1 \le j < k \le N} \frac{[z_j-z_k+1/2][w_k-w_j+1/2]}{[z_j-z_k][w_k-w_j]}
\nonumber \\
&\times \mathrm{det}_N (Y_N(z_1,\dots,z_N|w_1,\dots,w_N|h)).
\label{rightexpression}
\end{align}
To prove Theorem \ref{identitytheoremprove},
we first show the following properties for $L_N(z_1,\dots,z_N|w_1,\dots,w_N|h)$.

\begin{proposition} \label{korepinlemma}
The functions 
$L_N(z_1,\dots,z_N|w_1,\dots,w_N|h)$
satisfy, and are uniquely determined by, the following properties:
\begin{enumerate}
\item The functions
$L_N(z_1,\dots,z_N|w_1,\dots,w_N|h)$
are elliptic polynomials in $w_{N}$ of degree $2N$ with
the following quasi-periodicities: 
\begin{align}
&L_{N}(z_1,\dots,z_{N}|w_1,\dots,w_N+1|h)
=(-1)^{2N}
L_{N}(z_1,\dots,z_{N}|w_1,\dots,w_N|h)
, \label{qppartitionfunctionsone} \\
&L_{N}(z_1,\dots,z_N|w_1,\dots,w_N-i \log ({\bf q})/\pi|h)
 \nonumber\\
={}&(-{\bf q}^{-1})^{2N}
\exp (-2 \pi i
(2N w_N+h
)) L_{N}(z_1,\dots,z_{N}|w_1,\dots,w_N|h) \label{qppartitionfunctionstwo}.
\end{align}
\item The following relations hold:
\begin{align}
&\hspace{-36pt}L_{N}(z_1,\dots,z_{N}|w_1,\dots,w_N|h)
|_{w_N=-z_m}
=[h][-2z_m]
\prod_{ \substack{ j=1 \\ j \neq m} }^{N}[z_m-z_j+1/2][z_m+z_j]
\nonumber \\
\times&\prod_{j=1}^{N-1} [z_m+w_j+1/2][z_m-w_j]
L_{N-1}(z_1,\dots,\hat{z_m} ,\dots,z_{N}|w_1,\dots,w_{N-1}|h-1/2)
, \label{ordinaryrecursionwavefunction} \\
&\hspace{-36pt}L_{N}(z_1,\dots,z_{N}|w_1,\dots,w_N|h)
|_{w_N=z_m}
=[h][-2z_m]
\prod_{ \substack{ j=1 \\ j \neq m} }^{N}[z_m-z_j+1/2][z_m+z_j]
\nonumber \\
\times&\prod_{j=1}^{N-1} [z_m-w_j+1/2][z_m+w_j]
L_{N-1}(z_1,\dots,\hat{z_m} ,\dots,z_{N}|w_1,\dots,w_{N-1}|h+1/2)
, \label{ordinaryrecursionwavefunctiontwo}
\end{align}
for $m=1,\dots,N$,
and $\hat{z_m}$ in $L_{N-1}(z_1,\dots,\hat{z_m} ,\dots,z_{N}|w_1,\dots,w_{N-1}|h-1/2)$ and \\
$L_{N-1}(z_1,\dots,\hat{z_m} ,\dots,z_{N}|w_1,\dots,w_{N-1}|h+1/2)$
means that $z_m$ is removed.
\item  The following holds:
\begin{align}
&\hspace{-36pt} L_{1}(z_1|w_1|h)=
[h+z_1+w_1][w_1-z_1]-[h-z_1+w_1][w_1+z_1].
\label{ordinaryinitialrecursion}
\end{align}
\end{enumerate}
\end{proposition}

Proposition \ref{korepinlemma}
is a version of the so-called Korepin's Lemma in
the field of quantum integrable models \cite{Ko},
which list the properties of a sequence of functions
which uniquely define them.
Let us explain about the uniqueness.
Property 1 together with Proposition \ref{propositionelliptic}
means that $L_N(z_1,\dots,z_N|w_1,\dots,w_N|h)$
is uniquely determined by its evaluation at $2N$ points.
The evaluations at $2N$ points are Property 2,
which relates $L_N(z_1,\dots,z_N|w_1,\dots,w_N|h)$ at $z_N=\pm w_m$,
$(m=1,\dots,N)$ with
$L_{N-1}(z_1,\dots,\hat{z_m} ,\dots,z_{N}|w_1,\dots,w_{N-1}|h+1/2)$.
This means that $L_N(z_1,\dots,z_N|w_1,\dots,w_N|h)$
is uniquely determined from $L_{N-1}(z_1,\dots,z_{N-1}|w_1,\dots,w_{N-1}|h)$,
and Property 3 corresponds to the determination of the initial term 
of the sequence of functions
$\{ L_N(z_1,\dots,z_N|w_1,\dots,w_N|h) | N \in \mathbb{N} \}$.

Let us show Properties 1--3 in 
Proposition \ref{korepinlemma}.
We first expand
\eqref{leftexpression} as
\begin{align}
&L_N(z_1,\dots,z_N|w_1,\dots,w_N|h) \nonumber \\
=&
\prod_{1 \le j < k \le N} \frac{[z_j-z_k+1/2]}{[z_j-z_k]}
\sum_{\sigma \in S_N} \sum_{\tau_1,\dots,\tau_N=\pm 1}
\mathrm{sgn}(\sigma) (-1)^{|\tau|} 
\prod_{j=1}^N [h+(j-N)/2+\tau_{\sigma(j)} z_{\sigma(j)}+w_j] \nonumber \\
\times&\prod_{j=1}^N \prod_{\ell=1}^{j-1} [w_\ell+\tau_{\sigma(j)}z_{\sigma(j)}+1/2]
\prod_{j=1}^{N-1} \prod_{\ell=j+1}^{N} [w_\ell+\tau_{\sigma(j)} z_{\sigma(j)}]
\prod_{j=1}^N \prod_{\ell=1}^{N} [w_\ell-\tau_{\sigma(j)} z_{\sigma(j)}],
\label{useforproof}
\end{align}
where $|\tau|$ is the number of $\tau_j$'s ($j=1,\dots,N$) satisfying
$\tau_j=-1$.
Let us prove Property 1 
from the expansion \eqref{useforproof}.
One finds that each summand in \eqref{useforproof}
contains the following factors
\begin{align}
f_{\sigma,\tau}(w_N|z_1,\dots,z_N|h)=
[h+\tau_{\sigma(N)} z_{\sigma(N)}+w_N]
\prod_{j=1}^{N-1} [w_N+\tau_{\sigma(j)}z_{\sigma(j)}]
\prod_{j=1}^{N} [w_N-\tau_{\sigma(j)} z_{\sigma(j)}],
\end{align}
from which all the $w_N$-dependence come.
It is easy to calculate the quasi-periodicities for
$f_{\sigma,\tau}(w_N|z_1,\dots,z_N|h)$
\begin{align}
&f_{\sigma,\tau}(w_N+1|z_1,\dots,z_N|h)
=(-1)^{2N}
f_{\sigma,\tau}(w_N|z_1,\dots,z_N|h), \\
&f_{\sigma,\tau}(w_N-i \log ({\bf q})/\pi|z_1,\dots,z_N|h)
 \nonumber\\
={}&(-{\bf q}^{-1})^{2N}
\exp (-2 \pi i
(2N w_N+h
)) f_{\sigma,\tau}(w_N|z_1,\dots,z_N|h).
\end{align}
The quasi-periodicities do not depend on
$\sigma$ nor $\tau$, from which one finds
\eqref{qppartitionfunctionsone} and \eqref{qppartitionfunctionstwo}.

Now let us show Property 2.
First, we note that the determinant of the matrix \\
$X_N(z_1,\dots,z_N|w_1,\dots,w_N|h)$ whose matrix elements are given by
\eqref{matrixelementsone} is antisymmetric with respect to
$z_j \longleftrightarrow z_k$ ($j \neq k$).
The antisymmetry also holds for
$\displaystyle \prod_{1 \le j < k \le N} \frac{[z_j-z_k+1/2]}{[z_j-z_k]}$,
and since $L_N(z_1,\dots,z_N|w_1,\dots,w_N|h)$ is a product of
$\mathrm{det}_N (X_N(z_1,\dots,z_N|w_1,\dots,w_N|h))$
and $\displaystyle \prod_{1 \le j < k \le N} \frac{[z_j-z_k+1/2]}{[z_j-z_k]}$,
we find that $L_N(z_1,\dots,z_N|w_1,\dots,w_N|h)$ is symmetric with respect to
$z_j \longleftrightarrow z_k$ ($j \neq k$).
From this symmetry, it is enough to show
\eqref{ordinaryrecursionwavefunction} and
\eqref{ordinaryrecursionwavefunctiontwo} for the case $m=N$
\begin{align}
&\hspace{-36pt}L_{N}(z_1,\dots,z_{N}|w_1,\dots,w_N|h)
|_{w_N=-z_N}
=[h][-2z_N]
\prod_{j=1}^{N-1}[z_N-z_j+1/2][z_N+z_j]
\nonumber \\
\times&\prod_{j=1}^{N-1} [z_N+w_j+1/2][z_N-w_j]
L_{N-1}(z_1,\dots,z_{N-1}|w_1,\dots,w_{N-1}|h-1/2)
, \label{specialcasetoshow} \\
&\hspace{-36pt}L_{N}(z_1,\dots,z_{N}|w_1,\dots,w_N|h)
|_{w_N=z_N}
=[h][-2z_N]
\prod_{j=1}^{N-1}[z_N-z_j+1/2][z_N+z_j]
\nonumber \\
\times&\prod_{j=1}^{N-1} [z_N-w_j+1/2][z_N+w_j]
L_{N-1}(z_1,\dots,z_{N-1}|w_1,\dots,w_{N-1}|h+1/2)
. \label{specialcasetoshowtwo}
\end{align}
The other cases
\eqref{ordinaryrecursionwavefunction} and
\eqref{ordinaryrecursionwavefunctiontwo} for $m=1, \dots, N-1$
follows from \eqref{specialcasetoshow} and \eqref{specialcasetoshowtwo}
by using the property that $L_{N}(z_1,\dots,z_{N}|w_1,\dots,w_N|h)$
is a symmetric function with symmetric variables $z_j$ $(j=1,\dots,N)$.

Let us show \eqref{specialcasetoshow}.
After the substitution $w_N=-z_N$,
one finds that the only the summands satisfying $\sigma(N)=N$, $\tau_N=+1$
in \eqref{useforproof} survive.
Keeping this in mind, one rewrites $L_N(z_1,\dots,z_N|w_1,\dots,w_N|h)|_{w_N=-z_N}$ as a sum over $\sigma \in S_{N-1}$ and $\tau_1,\dots,\tau_{N-1}$ as
\begin{align}
&L_N(z_1,\dots,z_N|w_1,\dots,w_N|h)|_{w_N=-z_N} \nonumber \\
=&
\prod_{1 \le j < k \le N-1} \frac{[z_j-z_k+1/2]}{[z_j-z_k]}
\prod_{j=1}^{N-1} \frac{[z_j-z_N+1/2]}{[z_j-z_N]}
\sum_{\sigma \in S_{N-1}} \sum_{\tau_1,\dots,\tau_{N-1}=\pm 1}
\mathrm{sgn}(\sigma) (-1)^{|\tau|} \nonumber \\
\times& [h] \prod_{j=1}^{N-1} [h-1/2+(j-(N-1))/2+\tau_{\sigma(j)} z_{\sigma(j)}+w_j] \nonumber \\
\times&
\prod_{\ell=1}^{N-1} [w_\ell+z_{N}+1/2]
\prod_{j=1}^{N-1}
\prod_{\ell=1}^{j-1} [w_\ell+\tau_{\sigma(j)}z_{\sigma(j)}+1/2]
\nonumber \\
\times&
\prod_{j=1}^{N-1} [-z_N+\tau_{\sigma(j)} z_{\sigma(j)}]
\prod_{j=1}^{N-2} \prod_{\ell=j+1}^{N-1} [w_\ell+\tau_{\sigma(j)} z_{\sigma(j)}] \nonumber \\
\times&
[-2z_N]
\prod_{j=1}^{N-1} [-z_N-\tau_{\sigma(j)} z_{\sigma(j)}]
\prod_{\ell=1}^{N-1} [w_\ell- z_N]
\prod_{j=1}^{N-1} \prod_{\ell=1}^{N-1} [w_\ell-\tau_{\sigma(j)} z_{\sigma(j)}].
\label{tochu}
\end{align}
Using
\begin{align}
\prod_{j=1}^{N-1} [-z_N+\tau_{\sigma(j)} z_{\sigma(j)}]
\prod_{j=1}^{N-1} [-z_N-\tau_{\sigma(j)} z_{\sigma(j)}]
=\prod_{j=1}^{N-1} [-z_N+z_j][-z_N-z_j],
\end{align}
one can further rearrange \eqref{tochu} as
\begin{align}
&L_{N}(z_1,\dots,z_{N}|w_1,\dots,w_N|h)
|_{w_N=-z_N}
=[h][-2z_N]
\prod_{j=1}^{N-1}[z_N-z_j+1/2][z_N+z_j] \nonumber \\
\times&
\prod_{j=1}^{N-1} [z_N+w_j+1/2][z_N-w_j]
\prod_{1 \le j < k \le N-1} \frac{[z_j-z_k+1/2]}{[z_j-z_k]}
\nonumber \\
\times&
\sum_{\sigma \in S_{N-1}} \sum_{\tau_1,\dots,\tau_{N-1}=\pm 1}
\mathrm{sgn}(\sigma) (-1)^{|\tau|} 
\prod_{j=1}^{N-1} [h-1/2+(j-(N-1))/2+\tau_{\sigma(j)} z_{\sigma(j)}+w_j] \nonumber \\
\times&\prod_{j=1}^{N-1} \prod_{\ell=1}^{j-1} [w_\ell+\tau_{\sigma(j)}z_{\sigma(j)}+1/2]
\prod_{j=1}^{N-2} \prod_{\ell=j+1}^{N-1} [w_\ell+\tau_{\sigma(j)} z_{\sigma(j)}]
\prod_{j=1}^{N-1} \prod_{\ell=1}^{N-1} [w_\ell-\tau_{\sigma(j)} z_{\sigma(j)}]
\nonumber \\
=&[h][-2z_N]
\prod_{j=1}^{N-1}[z_N-z_j+1/2][z_N+z_j]
\prod_{j=1}^{N-1} [z_N+w_j+1/2][z_N-w_j] \nonumber \\
\times&L_{N}(z_1,\dots,z_{N-1}|w_1,\dots,w_{N-1}|h-1/2),
\end{align}
and we find \eqref{specialcasetoshow} holds.

\eqref{specialcasetoshowtwo} can be shown in a similar way.
In this case, one notes that the summands
satisfying $\sigma(N)=N$, $\tau_N=-1$
in \eqref{useforproof} survive after the substitution $w_N=z_N$.
Then one rewrites $L_N(z_1,\dots,z_N|w_1,\dots,w_N|h)|_{w_N=z_N}$ 
as a sum over $\sigma \in S_{N-1}$ and $\tau_1,\dots,\tau_{N-1}$
in the following way:
\begin{align}
&L_N(z_1,\dots,z_N|w_1,\dots,w_N|h)|_{w_N=z_N} \nonumber \\
=&
\prod_{1 \le j < k \le N-1} \frac{[z_j-z_k+1/2]}{[z_j-z_k]}
\prod_{j=1}^{N-1} \frac{[z_j-z_N+1/2]}{[z_j-z_N]}
\sum_{\sigma \in S_{N-1}} \sum_{\tau_1,\dots,\tau_{N-1}=\pm 1}
(-1) \mathrm{sgn}(\sigma) (-1)^{|\tau|} \nonumber \\
\times& [h] (-1)^{N-1} \prod_{j=1}^{N-1}[h+1/2+(j-(N-1))/2+\tau_{\sigma(j)} z_{\sigma(j)}+w_j] \nonumber \\
\times&
\prod_{\ell=1}^{N-1} [w_\ell-z_{N}+1/2]
\prod_{j=1}^{N-1}
\prod_{\ell=1}^{j-1} [w_\ell+\tau_{\sigma(j)}z_{\sigma(j)}+1/2]
\nonumber \\
\times&
\prod_{j=1}^{N-1} [z_N+\tau_{\sigma(j)} z_{\sigma(j)}]
\prod_{j=1}^{N-2} \prod_{\ell=j+1}^{N-1} [w_\ell+\tau_{\sigma(j)} z_{\sigma(j)}] \nonumber \\
\times&
[2z_N]
\prod_{j=1}^{N-1} [z_N-\tau_{\sigma(j)} z_{\sigma(j)}]
\prod_{\ell=1}^{N-1} [w_\ell+ z_N]
\prod_{j=1}^{N-1} \prod_{\ell=1}^{N-1} [w_\ell-\tau_{\sigma(j)} z_{\sigma(j)}]
\nonumber \\
=&[h][-2z_N]
\prod_{j=1}^{N-1}[z_N-z_j+1/2][z_N+z_j]
\nonumber \\
\times&\prod_{j=1}^{N-1} [z_N-w_j+1/2][z_N+w_j]
\prod_{1 \le j < k \le N-1} \frac{[z_j-z_k+1/2]}{[z_j-z_k]}
\nonumber \\
\times&
\sum_{\sigma \in S_{N-1}} \sum_{\tau_1,\dots,\tau_{N-1}=\pm 1}
\mathrm{sgn}(\sigma) (-1)^{|\tau|} 
\prod_{j=1}^{N-1} [h+1/2+(j-(N-1))/2+\tau_{\sigma(j)} z_{\sigma(j)}+w_j] \nonumber \\
\times&\prod_{j=1}^{N-1} \prod_{\ell=1}^{j-1} [w_\ell+\tau_{\sigma(j)}z_{\sigma(j)}+1/2]
\prod_{j=1}^{N-2} \prod_{\ell=j+1}^{N-1} [w_\ell+\tau_{\sigma(j)} z_{\sigma(j)}]
\prod_{j=1}^{N-1} \prod_{\ell=1}^{N-1} [w_\ell-\tau_{\sigma(j)} z_{\sigma(j)}]
\nonumber \\
=&[h][-2z_N]
\prod_{j=1}^{N-1}[z_N-z_j+1/2][z_N+z_j]
\prod_{j=1}^{N-1} [z_N-w_j+1/2][z_N+w_j] \nonumber \\
\times&L_{N}(z_1,\dots,z_{N-1}|w_1,\dots,w_{N-1}|h+1/2),
\end{align}
hence we have shown \eqref{specialcasetoshowtwo}.

The remaining thing to prove is Property 3,
which is obvious to see from the definition
of $L_N(z_1,\dots,z_N|w_1,\dots,w_N|h)$ \eqref{leftexpression}. \\

Next, we show that the functions $R_N(z_1,\dots,z_N|w_1,\dots,w_N|h)$
satisfy exactly the same properties in Proposition \ref{korepinlemma}.

\begin{proposition} \label{rightkorepinlemma}
The functions 
$R_N(z_1,\dots,z_N|w_1,\dots,w_N|h)$
satisfy, and are uniquely determined by, the following properties:
\begin{enumerate}
\item The functions
$R_N(z_1,\dots,z_N|w_1,\dots,w_N|h)$
are elliptic polynomials in $w_{N}$ of degree $2N$ with
the following quasi-periodicities: 
\begin{align}
&R_{N}(z_1,\dots,z_{N}|w_1,\dots,w_N+1|h)
=(-1)^{2N}
R_{N}(z_1,\dots,z_{N}|w_1,\dots,w_N|h)
, \label{rightqppartitionfunctionsone} \\
&R_{N}(z_1,\dots,z_N|w_1,\dots,w_N-i \log ({\bf q})/\pi|h)
 \nonumber\\
={}&(-{\bf q}^{-1})^{2N}
\exp (-2 \pi i
(2N w_N+h
)) R_{N}(z_1,\dots,z_{N}|w_1,\dots,w_N|h) \label{rightqppartitionfunctionstwo}.
\end{align}
\item The following relations hold:
\begin{align}
&\hspace{-36pt}R_{N}(z_1,\dots,z_{N}|w_1,\dots,w_N|h)
|_{w_N=-z_m}
=[h][-2z_m]
\prod_{ \substack{ j=1 \\ j \neq m} }^{N}[z_m-z_j+1/2][z_m+z_j]
\nonumber \\
\times&\prod_{j=1}^{N-1} [z_m+w_j+1/2][z_m-w_j]
R_{N-1}(z_1,\dots,\hat{z_m} ,\dots,z_{N}|w_1,\dots,w_{N-1}|h-1/2)
, \label{rightordinaryrecursionwavefunction} \\
&\hspace{-36pt}R_{N}(z_1,\dots,z_{N}|w_1,\dots,w_N|h)
|_{w_N=z_m}
=[h][-2z_m]
\prod_{ \substack{ j=1 \\ j \neq m} }^{N}[z_m-z_j+1/2][z_m+z_j]
\nonumber \\
\times&\prod_{j=1}^{N-1} [z_m-w_j+1/2][z_m+w_j]
R_{N-1}(z_1,\dots,\hat{z_m} ,\dots,z_{N}|w_1,\dots,w_{N-1}|h+1/2)
, \label{rightordinaryrecursionwavefunctiontwo}
\end{align}
for $m=1,\dots,N$,
and $\hat{z_m}$ in $R_{N-1}(z_1,\dots,\hat{z_m} ,\dots,z_{N}|w_1,\dots,w_{N-1}|h-1/2)$ and \\
$R_{N-1}(z_1,\dots,\hat{z_m} ,\dots,z_{N}|w_1,\dots,w_{N-1}|h+1/2)$
means that $z_m$ is removed.
\item  The following holds:
\begin{align}
&\hspace{-36pt} R_{1}(z_1|w_1|h)=
[h+z_1+w_1][w_1-z_1]-[h-z_1+w_1][w_1+z_1].
\label{rightordinaryinitialrecursion}
\end{align}
\end{enumerate}
\end{proposition}

Let us show Properties 1--3 of Proposition \ref{rightkorepinlemma}.
We first introduce the notation
\begin{align}
&c_N(z_1,\dots,z_N|w_1,\dots,w_N|h) \nonumber \\
=&\displaystyle \frac{[h]\prod_{j=1}^N [h-j/2+1]}
{[h-N/2 ]\prod_{j=1}^N [h+(N+1)/2-j]}
\prod_{1 \le j < k \le N} \frac{[z_j-z_k+1/2][w_k-w_j+1/2]}{[z_j-z_k][w_k-w_j]},
\end{align}
and write $R_N(z_1,\dots,z_N|w_1,\dots,w_N|h)$ as
\begin{align}
&R_N(z_1,\dots,z_N|w_1,\dots,w_N|h) \nonumber \\
=&c_N(z_1,\dots,z_N|w_1,\dots,w_N|h)
\mathrm{det}_N(Y_N(z_1,\dots,z_N|w_1,\dots,w_N|h)).
\end{align}
First, note that from the explicit form of the matrix elements of
$Y_N(z_1,\dots,z_N|w_1,\dots,w_N|h)$ \eqref{matrixelementstwo},
the dependence on $w_N$ in
$\mathrm{det}_N(Y_N(z_1,\dots,z_N|w_1,\dots,w_N|h))$
only comes from the $N$-th column, and the product of
factors which depend on $w_N$ in $c_N(z_1,\dots,z_N|w_1,\dots,w_N|h)$
is $\displaystyle \prod_{j=1}^{N-1} \frac{[w_N-w_j+1/2]}{[w_N-w_j]}$.
Then it is easy to check the quasi-periodicities
\eqref{rightqppartitionfunctionsone} and \eqref{rightqppartitionfunctionstwo}.
Note also that the factor $\displaystyle \prod_{j=1}^{N-1} \frac{[w_N-w_j+1/2]}{[w_N-w_j]}$ may lead to singularities at $w_N=w_j$, $j=1,\dots,N-1$,
but actually do not, since one can see from the matrix elements of
$Y_N(z_1,\dots,z_N|w_1,\dots,w_N|h)$ \eqref{matrixelementstwo}
that the determinant $\mathrm{det}_N(Y_N(z_1,\dots,z_N|w_1,\dots,w_N|h))$
vanishes when $w_j=w_k$ ($j \neq k$), and in particular
has zeroes at $w_N=w_j$, $j=1,\dots,N-1$.
$R_N(z_1,\dots,z_N|w_1,\dots,w_N|h)$ is holomorphic and
thus is an elliptic polynomial of $w_N$ of degree $2N$,
and Property 1 is proved.

Next we show Property 2. First, rewriting
$\mathrm{det}_N(Y_N(z_1,\dots,z_N|w_1,\dots,w_N|h))$ as
\begin{align}
&\mathrm{det}_N(Y_N(z_1,\dots,z_N|w_1,\dots,w_N|h))
=\prod_{k=1}^N \prod_{\ell=1}^{N} [z_\ell+w_k] [z_\ell-w_k] \nonumber \\
\times&\mathrm{det}_N ([h+(N-1)/2+w_k+z_j] [z_j+w_k]^{-1}
-[-h-(N-1)/2-w_k+z_j] [z_j-w_k]^{-1}),
\end{align}
one sees that $\mathrm{det}_N(Y_N(z_1,\dots,z_N|w_1,\dots,w_N|h))$
is antisymmetric with respect to
$z_j \longleftrightarrow z_k$ ($j \neq k$).
The factor 
$\displaystyle \prod_{1 \le j < k \le N} \frac{[z_j-z_k+1/2]}{[z_j-z_k]}$
in $c_N(z_1,\dots,z_N|w_1,\dots,w_N|h)$ is also
antisymmetric with respect to
$z_j \longleftrightarrow z_k$ ($j \neq k$), and we find that
$R_N(z_1,\dots,z_N|w_1,\dots,w_N|h)$ is symmetric with respect to
$z_j \longleftrightarrow z_k$ ($j \neq k$).
From this symmetry, it is enough to show
\eqref{rightordinaryrecursionwavefunction} and
\eqref{rightordinaryrecursionwavefunctiontwo} for the case $m=N$
\begin{align}
&\hspace{-36pt}R_{N}(z_1,\dots,z_{N}|w_1,\dots,w_N|h)
|_{w_N=-z_N}
=[h][-2z_N]
\prod_{j=1}^{N-1}[z_N-z_j+1/2][z_N+z_j]
\nonumber \\
\times&\prod_{j=1}^{N-1} [z_N+w_j+1/2][z_N-w_j]
R_{N-1}(z_1,\dots,z_{N-1}|w_1,\dots,w_{N-1}|h-1/2)
, \label{rightspecialcasetoshow} \\
&\hspace{-36pt}R_{N}(z_1,\dots,z_{N}|w_1,\dots,w_N|h)
|_{w_N=z_N}
=[h][-2z_N]
\prod_{j=1}^{N-1}[z_N-z_j+1/2][z_N+z_j]
\nonumber \\
\times&\prod_{j=1}^{N-1} [z_N-w_j+1/2][z_N+w_j]
R_{N-1}(z_1,\dots,z_{N-1}|w_1,\dots,w_{N-1}|h+1/2)
. \label{rightspecialcasetoshowtwo}
\end{align}
The other cases
\eqref{rightordinaryrecursionwavefunction} and
\eqref{rightordinaryrecursionwavefunctiontwo} for $m=1, \dots, N-1$
can be obtained from \eqref{rightspecialcasetoshow} and \eqref{rightspecialcasetoshowtwo}
by using the property that $R_{N}(z_1,\dots,z_{N}|w_1,\dots,w_N|h)$
is a symmetric function with symmetric variables $z_j$ $(j=1,\dots,N)$.

Let us show \eqref{rightspecialcasetoshow}.
First, we find the following relation for $c_N(z_1,\dots,z_N|w_1,\dots,w_N|h)$
\begin{align}
&c_N(z_1,\dots,z_N|w_1,\dots,w_N|h)|_{w_N=-z_N}
=\frac{[h][h+1/2]}{[h-1/2][h+N/2-1/2]} \nonumber \\
\times&\prod_{j=1}^{N-1} \frac{[z_N+w_j-1/2]}{[z_N+w_j]}
\prod_{j=1}^{N-1} \frac{[z_j-z_N+1/2]}{[z_j-z_N]}
c_{N-1}(z_1,\dots,z_{N-1}|w_1,\dots,w_{N-1}|h-1/2).
\label{factoruse}
\end{align}
Next, we analyze the determinant $\mathrm{det}_N(Y_N(z_1,\dots,z_N|w_1,\dots,w_N|h))$. It can be easily seen from the explicit form of the matrix elements of
$Y_N(z_1,\dots,z_N|w_1,\dots,w_N|h)$ \eqref{matrixelementstwo} that
among the matrix elements in
the $N$-th column of $Y_N(z_1,\dots,z_N|w_1,\dots,w_N|h)$,
only the matrix element in
$N$-th row is nonzero after the substitution $w_N=-z_N$.
Then one expands the determinant $\mathrm{det}_N(Y_N(z_1,\dots,z_N|w_1,\dots,w_N|h))|_{w_N=-z_N}$ by its $N$-th column to get
\begin{align}
&\mathrm{det}_N(Y_N(z_1,\dots,z_N|w_1,\dots,w_N|h))|_{w_N=-z_N} \nonumber \\
=&[h+(N-1)/2] \prod_{\ell=1}^{N-1} [z_\ell-z_N] \prod_{\ell=1}^N [z_\ell+z_N]
\mathrm{det}_{N-1}(\overline{Y}_N(z_1,\dots,z_{N}|w_1,\dots,w_{N}|h))
\label{righttochu},
\end{align}
where $\overline{Y}_N(z_1,\dots,z_{N}|w_1,\dots,w_{N}|h)$
is an $(N-1) \times (N-1)$ matrix which is obtained from
$Y_N(z_1,\dots,z_{N}|w_1,\dots,w_{N}|h)$ by removing the $N$-th row and
the $N$-th column. Since one can show
\begin{align}
&\overline{Y}_N(z_1,\dots,z_{N}|w_1,\dots,w_{N}|h)_{jk} \nonumber \\
=&Y_N(z_1,\dots,z_{N}|w_1,\dots,w_{N}|h)_{jk} \nonumber \\
=&-[z_N+w_k][z_N-w_k]
Y_{N-1}(z_1,\dots,z_{N-1}|w_1,\dots,w_{N-1}|h-1/2)_{jk},
\end{align}
for $j,k=1,\dots,N-1$, we can further rewrite \eqref{righttochu} as
\begin{align}
&\mathrm{det}_N(Y_N(z_1,\dots,z_N|w_1,\dots,w_N|h))|_{w_N=-z_N} \nonumber \\
=&[h+(N-1)/2] \prod_{\ell=1}^{N-1} [z_\ell-z_N] \prod_{\ell=1}^N [z_\ell+z_N]
\nonumber \\
\times&\prod_{k=1}^{N-1} (-[z_N+w_k][z_N-w_k])
\mathrm{det}_{N-1}(Y_{N-1}(z_1,\dots,z_{N-1}|w_1,\dots,w_{N-1}|h-1/2))
\nonumber \\
=&[2z_N][h+(N-1)/2] \prod_{j=1}^{N-1} [z_j-z_N][z_j+z_N]
\prod_{j=1}^{N-1} [z_N+w_j][w_j-z_N] \nonumber \\
\times&\mathrm{det}_{N-1}(Y_{N-1}(z_1,\dots,z_{N-1}|w_1,\dots,w_{N-1}|h-1/2))
\label{righttochutwo}.
\end{align}
Combining \eqref{factoruse} and \eqref{righttochutwo}, we get
\begin{align}
&R_N(z_1,\dots,z_N|w_1,\dots,w_N|h))|_{w_N=-z_N} \nonumber \\
=&c_N(z_1,\dots,z_N|w_1,\dots,w_N|h)
\mathrm{det}_N(Y_N(z_1,\dots,z_N|w_1,\dots,w_N|h))|_{w_N=-z_N} \nonumber \\
=&\frac{[h][h+1/2]}{[h-1/2][h+N/2-1/2]}
\prod_{j=1}^{N-1} \frac{[z_N+w_j-1/2]}{[z_N+w_j]}
\prod_{j=1}^{N-1} \frac{[z_j-z_N+1/2]}{[z_j-z_N]} \nonumber \\
\times&[2z_N][h+(N-1)/2] \prod_{j=1}^{N-1} [z_j-z_N][z_j+z_N]
\prod_{j=1}^{N-1} [z_N+w_j][w_j-z_N] \nonumber \\
\times&c_{N-1}(z_1,\dots,z_{N-1}|w_1,\dots,w_{N-1}|h-1/2)
\mathrm{det}_{N-1}(Y_{N-1}(z_1,\dots,z_{N-1}|w_1,\dots,w_{N-1}|h-1/2))
\nonumber \\
=&[h][-2z_N]
\prod_{j=1}^{N-1}[z_N-z_j+1/2][z_N+z_j]
\prod_{j=1}^{N-1} [z_N+w_j+1/2][z_N-w_j] \nonumber \\
\times&R_{N-1}(z_1,\dots,z_{N-1}|w_1,\dots,w_{N-1}|h-1/2),
\end{align}
which is exactly the relation \eqref{rightspecialcasetoshow}.

\eqref{rightspecialcasetoshowtwo} can be proved in a similar way.
First, we note the following relation holds
\begin{align}
&c_N(z_1,\dots,z_N|w_1,\dots,w_N|h)|_{w_N=z_N}
=-\frac{[h-N/2+1]}{[h+1/2]} \nonumber \\
\times&\prod_{j=1}^{N-1} \frac{[z_N-w_j+1/2]}{[z_N-w_j]}
\prod_{j=1}^{N-1} \frac{[z_j-z_N+1/2]}{[z_j-z_N]}
c_{N-1}(z_1,\dots,z_{N-1}|w_1,\dots,w_{N-1}|h+1/2).
\label{factorusetwo}
\end{align}
Next, one can see that only the matrix element in the $N$-th row
is nonzero among the matrix elements
in the $N$-th column of $Y_N(z_1,\dots,z_N|w_1,\dots,w_N|h)$,
and we find that the expansion of the determinant
$\mathrm{det}_N(Y_N(z_1,\dots,z_N|w_1,\dots,w_N|h))|_{w_N=z_N}$
by its $N$-th column gives the following relation
\begin{align}
&\mathrm{det}_N(Y_N(z_1,\dots,z_N|w_1,\dots,w_N|h))|_{w_N=z_N} \nonumber \\
=&[h+(N-1)/2] \prod_{\ell=1}^{N-1} [z_\ell-z_N] \prod_{\ell=1}^N [z_\ell+z_N]
\mathrm{det}_{N-1}(\overline{Y}_N(z_1,\dots,z_{N}|w_1,\dots,w_{N}|h))
\label{righttochuthree}.
\end{align}
Using the relation
\begin{align}
&\overline{Y}_N(z_1,\dots,z_{N}|w_1,\dots,w_{N}|h)_{jk} \nonumber \\
=&Y_N(z_1,\dots,z_{N}|w_1,\dots,w_{N}|h)_{jk} \nonumber \\
=&[z_N+w_k][z_N-w_k]
Y_{N-1}(z_1,\dots,z_{N-1}|w_1,\dots,w_{N-1}|h+1/2)_{jk},
\end{align}
for $j,k=1,\dots,N-1$, we can further rewrite \eqref{righttochuthree} as
\begin{align}
&\mathrm{det}_N(Y_N(z_1,\dots,z_N|w_1,\dots,w_N|h))|_{w_N=z_N} \nonumber \\
=&[h+(N-1)/2] \prod_{\ell=1}^{N-1} [z_\ell-z_N] \prod_{\ell=1}^N [z_\ell+z_N]
\nonumber \\
\times&\prod_{k=1}^{N-1} [z_N+w_k][z_N-w_k]
\mathrm{det}_{N-1}(Y_{N-1}(z_1,\dots,z_{N-1}|w_1,\dots,w_{N-1}|h+1/2))
\nonumber \\
=&[2z_N][h+(N-1)/2] \prod_{j=1}^{N-1} [z_j-z_N][z_j+z_N]
\prod_{j=1}^{N-1} [z_N+w_j][z_N-w_j] \nonumber \\
\times&\mathrm{det}_{N-1}(Y_{N-1}(z_1,\dots,z_{N-1}|w_1,\dots,w_{N-1}|h+1/2))
\label{righttochufour}.
\end{align}
Combining \eqref{factorusetwo} and \eqref{righttochufour}, we get
\begin{align}
&R_N(z_1,\dots,z_N|w_1,\dots,w_N|h))|_{w_N=z_N} \nonumber \\
=&c_N(z_1,\dots,z_N|w_1,\dots,w_N|h)
\mathrm{det}_N(Y_N(z_1,\dots,z_N|w_1,\dots,w_N|h))|_{w_N=z_N} \nonumber \\
=&
-\frac{[h-N/2+1]}{[h+1/2]} \prod_{j=1}^{N-1} \frac{[z_N-w_j+1/2]}{[z_N-w_j]}
\prod_{j=1}^{N-1} \frac{[z_j-z_N+1/2]}{[z_j-z_N]}
\nonumber \\
\times&[2z_N][h+(N-1)/2] \prod_{j=1}^{N-1} [z_j-z_N][z_j+z_N]
\prod_{j=1}^{N-1} [z_N+w_j][z_N-w_j]
\nonumber \\
\times&c_{N-1}(z_1,\dots,z_{N-1}|w_1,\dots,w_{N-1}|h+1/2)
\mathrm{det}_{N-1}(Y_{N-1}(z_1,\dots,z_{N-1}|w_1,\dots,w_{N-1}|h+1/2))
\nonumber \\
=&
[h][-2z_N]
\prod_{j=1}^{N-1}[z_N-z_j+1/2][z_N+z_j]
\prod_{j=1}^{N-1} [z_N-w_j+1/2][z_N+w_j]
\nonumber \\
\times&R_{N-1}(z_1,\dots,z_{N-1}|w_1,\dots,w_{N-1}|h+1/2),
\end{align}
hence the relation \eqref{rightspecialcasetoshowtwo} is proved.
Note that in the last equality, we used
the identity $\displaystyle \frac{[h+(N-1)/2][h-N/2+1]}{[h+1/2]}=[h]$
which holds for any integer $N$.

What remains is to show Property 3,
which can be easily seen from the definition of
$R_N(z_1,\dots,z_N|w_1,\dots,w_N|h))$ \eqref{rightexpression}. \\

Finally, the two propositions we proved
(Propositions \ref{korepinlemma} and \ref{rightkorepinlemma})
mean that the sequence of functions
$\{ L_N(z_1,\dots,z_N|w_1,\dots,w_N|h) | N \in \mathbb{N} \}$
and
$\{ R_N(z_1,\dots,z_N|w_1,\dots,w_N|h) | N \in \mathbb{N} \}$
are exactly the same, and hence
\begin{align}
L_N(z_1,\dots,z_N|w_1,\dots,w_N|h)=R_N(z_1,\dots,z_N|w_1,\dots,w_N|h),
\end{align}
for $N \in \mathbb{N}$.
This concludes the proof of Theorem \ref{identitytheoremprove}.

\end{proof}

\section{Conclusion}
In this paper, we proved a duality between two elliptic determinants.
The proof presented in this paper is inspired by and
can be regarded as a variant of the Izergin-Korepin method.
It is originally a method initiated by Korepin and Izergin \cite{Ko,Iz}
to study and find explicit forms
of partition functions of quantum integrable models.
The key of the Izergin-Korepin method is to list the
properties for a sequence of functions which uniquely define them,
and one can use this idea to prove identities between
determinants which do not seem to have factorized expressions
and which look different at first sight.

It seems that there are many other transformation formulas
between (elliptic) determinants or Pfaffians
which do not seem to have factorized expressions,
and it is interesting to discover and prove them. The previous studies on
factorized formulas for (elliptic) determinants and Pfaffians
\cite{Cauchy,Frob,Kr,Has,TV,War,RosSch,Schpath,BK,Rains,Okadapfaffian,Roselldet,Rosellpfaffian} may give hints to find them.
Another interesting resource for the discovery is
partition functions of (elliptic) integrable models.
We found a duality between two elliptic Pfaffians as
a consequence of analyzing a variant of the domain wall
boundary partition functions (OS boundary) \cite{MotegiPfaffian}.
The partition functions of elliptic integrable models
may also give clues to find them.
As for the trigonometric $U_q(\widehat{sl_2})$ six-vertex model,
Kuperberg \cite{Ku2} uses various variations
of the domain wall boundary partition functions
to compute various generating functions of
the enumeration of alternating sign matrices.
We lifted one of his variations from the trigonometric model
to the elliptic model in \cite{MotegiPfaffian} and found
a duality between two elliptic Pfaffians.
It may also be interesting to lift other variations
of the domain wall boundary partition functions
to the elliptic model and find transformation formulas.

\section*{Acknowledgments}
This work was partially supported by Grant-in-Aid
for Scientific Research (C)
No. 18K03205 and No. 16K05468.


\begin{thebibliography}{}
\bibitem{Cauchy}
Cauchy, A.L.:
Memoire sur les fonctions alternees et sur les sommes alternees.
Exercices Anal. et Phys. Math. {\bf 2}, 151-159 (1841)
%
\bibitem{Frob}
Frobenius, F.:
Uber die elliptischen Funktionen zweiter Art.
J. fur die reine und ungew. Math.
{\bf 93}, 53-68 (1882)
%
\bibitem{Kr}
Krattenthaler, C.:
Advanced determinant calculus, S\'emin. Lothar. Comb.
{\bf 42}, B42q (1999)
%
\bibitem{Has}
Hasegawa, K.:
Ruijsenaars' commuting difference operators as commuting transfer matrices.
Commun. Math. Phys. {\bf 187}, 289-325 (1997)
%
\bibitem{TV}
Tarasov, V., Varchenko, A.:
Geometry of $q$-hypergeometric functions, quantum affine algebras
and elliptic quantum groups.
Ast\'erisque {\bf 246} (1997)
%
\bibitem{War}
Warnaar, O.:
Summation and transformation formulas for elliptic hypergeometric series.
Constr. Approx. {\bf 18}, 479-502 (2002)
%
\bibitem{RosSch}
Rosengren, H., Schlosser, M.:
Elliptic determinant evaluations and the Macdonald identities for affine root systems.
Comp. Math. {\bf 142}, 937-961 (2006)
%
\bibitem{Schpath}
Schlosser, M.:
Elliptic enumeration of nonintersecting lattice paths.
J. Combin. Theory Ser. A {\bf 114}, 505-521 (2007)
%
\bibitem{BK}
Bhatnagar, G., Krattenthaler, C.:
The Determinant of an Elliptic Sylvesteresque Matrix.
SIGMA {\bf 14}, 052, 15pp (2018)
%
\bibitem{Rains}
Rains, E.:
Recurrences for elliptic hypergemetric integrals.
Rokko Lectures in Mathematics {\bf 18},
Elliptic Integrable Systems:
183-199 (2005)
%
\bibitem{Okadapfaffian}
Okada, S.:
An elliptic generalization of Schur's Pfaffian identity.
Adv. Math. {\bf 204}, 530-538 (2006)
%
\bibitem{Roselldet}
Rosengren, H.:
Sums of triangular numbers from the Frobenius determinant.
Adv. Math. {\bf 208}, 935-961 (2007)
%
\bibitem{Rosellpfaffian}
Rosengren, H.:
Sums of squares from elliptic pfaffians.
Int. J. Number Theory {\bf 4}, 873-902 (2008)
%
\bibitem{Schur}
Schur, I.:
Uber die Darstellung der symmetrischen und der alternirenden Gruppe durch gebrochene lineare Substitutuionen.
J. Reine Angew. Math. {\bf 139}, 155-250 (1911)
%
\bibitem{RosRokko}
Rosengren, H.:
An elliptic determinant transformation.
Rokko Lectures in Mathematics {\bf 18},
Elliptic Integrable Systems:
241-246 (2005)
%
\bibitem{RosSchtwo}
Rosengren, H., Schlosser, M.:
Summations and transformations for multiple basic and elliptic
hypergeometric series by determinant evaluations.
Indag. Math. {\bf 14}, 483-514 (2003)
%
\bibitem{MotegiPfaffian}
Motegi, K.:
Elliptic free-fermion model with OS boundary and elliptic Pfaffians.
Lett. Math. Phys. arXiv:1806.00933.
%
\bibitem{DJKMO}
Date, E., Jimbo, M., Kuniba, A., Miwa, T., Okado, M.:
Exactly solvable SOS models. II. Proof of the star-triangle relation and combinatorial identities.
In: Conformal Field Theory and Solvable Lattice Models (Kyoto, 1986),
Adv. Stud. Pure Math., vol. 16, pp. 17-122. Academic Press, Boston, MA (1988)
%
\bibitem{FT}
Frenkel, I.B., Turaev, V.G.:
Elliptic solutions of the Yang-Baxter equation
and modular hypergeometric functions.
In: The Arnold-Gelfand Mathematical Seminars,
pp. 171-204 Birkhauser Boston, Boston, MA (1997)
%
\bibitem{Rosroot}
Rosengren, H.:
Elliptic hypergeometric series on root systems. 
Adv. Math. {\bf 181}, 417-447 (2004)
%
\bibitem{CosGus}
Coskun, H., Gustafson, R.A.:
Well-poised Macdonald functions $W_\lambda$ and Jackson coefficients
$\omega_\lambda$ on $BC_n$.
In: Jack, Hall-Littlewood and Macdonald polynomials, Contemp. Math.
{\bf 417}, Amer. Math. Soc. V.B. Kuznetsov and S. Sahi(eds.),
pp. 127-155 (2006)
%
\bibitem{RainsBC}
Rains, E.M.: 
$BC_{n}$-symmetric abelian functions. Duke Math. J. {\bf 135},
99-180 (2006)
%
\bibitem{Rainselliptic}
Rains, E.M.:
Transformations of elliptic hypergeometric integrals. Ann. of Math. {\bf 171},
169-243
(2010)
%
\bibitem{Spi}
Spiridonov, V.P.:
Theta hypergeometric integrals. St. Petersburg Math. J. {\bf 15},
929-967 (2003)
%
\bibitem{vDSpi}
van Diejen, J.F., Spiridonov, V.P.:
Elliptic Selberg integrals. Int. Math. Res. Not. {\bf 2001}, 1083-1110 (2001)
%
\bibitem{Rossummation}
Rosengren, H.:
Gustafson-Rakha-Type Elliptic Hypergeometric Series.
SIGMA {\bf 13}, 037, 11pp (2017)
%
\bibitem{ItoNoumi}
Ito, M., Noumi, M.:
Derivation of a $BC_n$
elliptic summation formula via the fundamental invariants.
Constr. Approx., {\bf 45}, 33-46 (2017)
%
\bibitem{RW}
Rosengren, H., Warnaar, S.O.:
Elliptic hypergeometric functions associated with root systems.
``Multivariable Special Functions".
arXiv:1704.08406
%
\bibitem{Ko}
Korepin, V.E.: Calculation of norms of Bethe wave functions. Commun. Math. Phys. {\bf 86}, 391-418
(1982)
%
\bibitem{Iz}
Izergin, A.: Partition function of the six-vertex model in a finite volume. Sov. Phys. Dokl. {\bf 32}, 878-879
(1987)
%
\bibitem{Dr}
Drinfeld, V.:
Hopf algebras and the quantum Yang-Baxter equation. Sov. Math.
Dokl. {\bf 32}, 254-258 (1985)
%
\bibitem{J}
Jimbo, M.:
A $q$-difference analogue of $U(G)$ and the Yang-Baxter equation. Lett.
Math. Phys. {\bf 10}, 63-69 (1985)
%
\bibitem{FRT}
Reshetikhin, N.Y., Takhtajan, L.A., Faddeev, L.D.:
Quantization of Lie groups and Lie algebras.
Leningrad Math. J. {\bf 1}, 193-225 (1990)
%
\bibitem{Baxter}
Baxter, R.J.: Exactly Solved Models in Statistical Mechanics. Academic Press, London
(1982)
%
\bibitem{KBI}
Korepin, V.B., Bogoliubov, N.M., Izergin, A.G.: Quantum Inverse Scattering Method
and Correlation Functions. Cambridge University Press, Cambridge (1993)
%
\bibitem{LW}
Lieb, E.H., Wu, F.Y.: Two-Dimensional Ferroelectric Models. In: Phase Transitions and Critical Phenomena,
vol. 1, pp. 331-490. Academic Press, London (1972)
%
\bibitem{Br}
Bressoud, D.:
Proofs and Confirmations:
The Story of the Alternating Sign Matrix Conjecture.
MAA Spectrum,
Mathematical Association of America,
Washington, DC (1999)
%
\bibitem{Ku1}
Kuperberg, G.:
Another proof of the alternating-sign matrix conjecture.
Int. Math. Res. Not. {\bf 3}, 139-150 (1996)
%
\bibitem{Ku2}
Kuperberg, G.:
Symmetry classes of alternating-sign matrices under one roof.
Ann. Math. {\bf 156}, 835-866 (2002)
%
\bibitem{Okada}
Okada, S.:
Enumeration of symmetry classes of alternating sign matrices and characters
of classical groups.
J. Alg. Comb. {\bf 23}, 43-69 (2001)
%
\bibitem{CP}
Colomo, F., Pronko, A.G.:
Square ice, alternating sign matrices, and classical
orthogonal polynomials.
J. Stat. Mech.: Theor. Exp. P01005 (2005)
%
\bibitem{BWZ}
Betea, D., Wheeler, M., Zinn-Justin, P.:
Refined Cauchy/Littlewood identities and six-vertex model partition functions: II. Proofs and new conjectures.
J. Alg. Combinatorics. {\bf 42}, 555-603 (2015) 
%
\bibitem{KZ}
Korepin, V., Zinn-Justin, O.:
Thermodynamic limit of the six-vertex model with domain wall boundary conditions.
J. Phys. A: Math. Gen. {\bf 33}, 7053 (2002)
%
\bibitem{Tsuchiya}
Tsuchiya, O.:
Determinant formula for the six-vertex model with reflecting end.
J. Math. Phys. {\bf 39}, 5946-5951 (1998)
%
\bibitem{Wheeler}
Wheeler, M.:
An Izergin-Korepin procedure for calculating scalar products in the six-vertex model. 
Nucl. Phys. B {\bf 852},  469-507 (2011)
%
\bibitem{MotegiIzerginKorepin}
Motegi, K.:
Symmetric functions and wavefunctions of XXZ-type six-vertex models and
elliptic Felderhof models by Izergin-Korepin analysis.
J. Math. Phys. {\bf 59}, 053505 (2018)
%
\bibitem{PRS}
Pakuliak, S., Rubtsov, V., Silantyev, A.:
The SOS model partition function and the elliptic weight
functions.
J. Phys. A: Math. Theor. {\bf 41}, 295204 (2008)
%
\bibitem{FSfelderhof}
Felder, G.: and A. Schorr, A.:
Separation of variables for quantum integrable systems on elliptic curves.
J. Phys. A: Math. Gen. {\bf 32}, 8001 (1999)
%
\bibitem{ABFfelderhof}
Andrews, G.E., Baxter, R.J., Forrester, P.J.:
Eight-vertex SOS model and generalized Rogers-Ramanujan-type identities.
J. Stat. Phys. {\bf 35}, 193-266 (1984)
%
\bibitem{Ros}
Rosengren, H.:
An Izergin-Korepin-type identity for the 8VSOS model, with applications to
alternating sign matrices
Adv. Appl. Math. {\bf 43}, 137-155 (2009)
%
\bibitem{YZ}
Yang, W.-L., Zhang, Y.-Z.:
Partition function of the eight-vertex model
with domain wall boundary condition.
J. Math. Phys. {\bf 50}, 083518 (2009)

\end{thebibliography}
\end{document}